\documentclass[11pt,leqno,oneside]{amsart}
\usepackage[width=6.5in,height=8.8in]{geometry}
\usepackage[mathscr]{eucal}
\usepackage{amssymb, amsmath,array, amscd, bbm, comment, charter}
\usepackage{enumerate}
\usepackage{graphicx}
\usepackage{url}
\usepackage[colorlinks,plainpages,backref]{hyperref}

\newtheorem{thm}{Theorem}[section]

\newtheorem{lem}[thm]{Lemma}

\newtheorem{pro}[thm]{Proposition}
\newtheorem{ex}[thm]{Example}
\theoremstyle{definition}
\newtheorem{rem}[thm]{Remark}

\begin{document}
\title{Upper level sets of Lelong numbers on $\mathbb P^2$ and cubic curves}

\author{Al\.i Ula\c{s} \"Ozg\"ur K\.i\c{s}\.isel}
\address{Department  of Mathematics,  Middle East  Technical University,  06800 Ankara,  Turkey }
\email{akisisel@metu.edu.tr}

\address{Department  of Mathematics,  Middle East  Technical University,  06800 Ankara,  Turkey }
\email{oyazici@metu.edu.tr}

\author{Ozcan Yazici}
\address{} 
\email{}
\date{\today}

\begin{abstract} Let $T$ be a positive closed current of bidimension $(1,1)$ with unit mass on $\mathbb P^2$ and $V_{\alpha}(T)$ be the upper level sets of Lelong numbers $\nu(T,x)$ of $T$. For any $\alpha\geq \frac{1}{3}$, we show that $|V_{\alpha}(T)\setminus C|\leq 2$ for some cubic curve $C$.  \\


 \end{abstract}
\maketitle
\section{Introduction}

Suppose that $T$ is a positive closed current of bidimension $(1,1)$ on the complex projective plane $\mathbb{P}^2$ such that $$||T||=\int_{\mathbb P^2}T\wedge \omega=1$$ where $\omega$ is the Fubini-Study form on $\mathbb P^2$. The Lelong number of $T$ at a point $x\in \mathbb{P}^2$ is defined  by  $$\displaystyle \nu(T,x)=\lim_{r\to 0}\frac{1}{\pi r^2}\int_{B(x,r)}T\wedge \omega$$ where $B(x,r)$ is the Euclidean ball with center $x$ and  radius $r>0$ (See \cite{Dem} for equivalent definitions and properties of Lelong numbers).

For $0<\alpha<1$, upper level sets of Lelong numbers $\nu(T,.)$ are defined by  
\begin{eqnarray*} 
E_{\alpha}(T):= \{x\in \mathbb{P}^{2} | \nu(T,x)\geq \alpha \}  \\ 
V_{\alpha}(T):= \{x\in \mathbb{P}^{2} | \nu(T,x) > \alpha \}.  \\ 
\end{eqnarray*} 
It is a remarkable result of Siu (\cite{Siu}) that  $E_{\alpha}(T)$ is   an analytic subvariety of $\mathbb P^2$ of dimension at most $1$ for any $\alpha >0$. It follows from Chow's theorem that $E_{\alpha}(T)$ is an algebraic variety. 

In \cite{Com}, Coman proved that if $\alpha \geq \frac{1}{2}$, $|V_{\alpha}(T)\setminus L|\leq 1$ for some complex line $L$  and if $\alpha \geq \frac{2}{5}$, $|V_{\alpha}(T)\setminus C|\leq 1$ for some conic. In \cite{Com, CT, Hef1, Hef2}, further geometric results are obtained for the sets $V_{\alpha}(T)$ where $T$ is a positive closed current of bidimension $(p,p)$ in $\mathbb P^n$ and $p\geq1.$  
Our goal here is to understand the geometry of level sets $V_{\alpha}(T)$ for smaller $\alpha>0$. In this direction, we show that for any $\alpha \geq \frac{1}{3}$, $|V_{\alpha}(T)\setminus C|\leq 2$ for some cubic curve $C$. More precisely, our main result is the following:

\begin{thm} \label{0}
If $\alpha \geq 1/3$ then either
\begin{itemize} 
\item  $|V_{\alpha}(T)\setminus C|\leq 1$ for some (possibly reducible) conic,    
\item or $V_{\alpha}(T)$ is a finite set and $|V_{\alpha}(T) \setminus C|\leq 2 $ for some cubic $C$. 
\end{itemize} 
\end{thm} 

The proof of our main result is based on the construction of entire plurisubharmonic functions of logarithmic growth and with weighted  logarithmic poles at some points of $\mathbb C^2$ (See Lemma \ref{2}). In doing so, we followed the ideas in \cite{Com}, where such plurisubharmonic functions with smaller number of poles were constructed. 
\section{Preliminaries}

We say that a plurisubharmonic function $u\in PSH(\mathbb C^n)$ has logarithmic growth if $$\gamma_u=\lim_{||z||\to\infty} \sup\frac{u(z)}{\log||z||}<\infty.  $$

A plurisubharmonic function $u\in PSH(\mathbb C^n)$ is said to have logarithmic pole of weight $\alpha>0$ at $p\in \mathbb C^n$ if $$u(z)=\alpha\log||z-p ||+\mathcal O(1)$$ in an open neighborhood of $p$.  

In the following example, we show that the number $1/3$ in the theorem is optimal.
\begin{ex} 
Suppose that $0<\alpha<1/3$. Consider $6$ generic lines $L_{1},\ldots,L_{6}$ in $\mathbb{P}^2$ and consider the set $S$ of $15$ points of pairwise intersections of these lines. Let 
\[ T:= \frac{1}{6} \sum_{j=1}^{6} [L_{j}]  \]
Then, clearly $||T||=1$ . For each $p\in S$, we have $\nu(T,p)=1/3$. Therefore, since $\alpha<1/3$, we must have $p\in V_{\alpha}(T)$ and consequently $S\subseteq V_{\alpha}(T)$. Let us now show that $V_{\alpha}(T)$ cannot be contained in the union of a cubic and two points: Let $S=\{x_1,\dots,x_{15}\}$ and $\Gamma$ be a cubic containing $S$ except two points, say $x_1,x_2$. We may assume that  $L_1$ and $L_2$ do not contain any of the points $x_1$ and $x_2$. Then Bezout's theorem shows that $C=L_1\cup L_2 \cup L$ for some line $L$. But this implies that the $4$ points in $S\setminus (L_1\cup L_2\cup \{x_1,x_2\})$ are collinear, which contradicts to the configuration of points of $S$.   
\end{ex}

The following proposition plays an important role in our proof of the main theorem.
 
\begin{pro}[Proposition 2.1 in \cite{Com}] \label{prop1} Let $S=\{x_1,\dots,x_k \}\subset \mathbb C^n$ and $T$ be a positive closed current on $\mathbb P^n$ of bidimension $(1,1)$ with unit mass. If $u\in PSH(\mathbb C^n)$ has logarithmic growth, it is locally bounded outside a finite set, and $u(z)\leq\alpha_j\log||z-x_j ||+\mathcal O(1)$ for $z$ near $x_j$ where $\alpha_j>0$, $1\leq j\leq k$, then $$\sum_{j}^k\alpha_j\nu(T,x_j)\leq \gamma_u. $$
\end{pro}
 In order to construct a plurisubharmonic function $u$ as in Proposition \ref{prop1}, we will use the following result of Coman-Nivoche.
 
 \begin {thm}[Theorem 4.1 in \cite{CN}]\label{cn} Let $r$ be a positive integer and $P,Q$ be polynomials such that $S=\{z\in \mathbb C^2: P(z)=Q(z)=0 \}$ and for any $x\in S$, $ord(P,x)\geq r$,   $ord(Q,x)\geq r$  and the intersection number $(P.Q)_x=r^2.$ Then $$u=\frac{1}{2r}\log(|P|^2+|Q|^2)$$ is a plurisubharmonic function in $\mathbb C^2$ with logarithmic growth, which is locally bounded outside of $S$ and has logarithmic pole of weight $1$ at the points of $S$. 
 
 \end{thm}

\section{Proof of the Main Result}  

Assume from now on that $S$ is a subset of $\mathbb{P}^2$ such that $|S|=12$  where $|S|$  denotes the number of points in $S$. Say $S=\{x_{1},\ldots,x_{12}\}$. Let $\mathcal{X}$ be the subspace of $\mathbb{C}[X,Y,Z]$ containing all homogeneous polynomials $P(X,Y,Z)$ of degree $6$ such that $P$ vanishes on all points of $S$, and the first partial derivatives of $P$ vanish on $x_{1},\ldots,x_{6}$. Each of these vanishings imposes a linear condition on the coefficients of $P$; the totality of these conditions may or may not be linearly independent. We can then estimate the dimension of $\mathcal{X}$ from below: 
\[  dim(\mathcal{X})\geq {6+2 \choose 2}-(12+2\times 6) = 4 \]  

Recall that the space of all homogeneous cubic polynomials in $\mathbb{C}[X,Y,Z]$ has dimension 10, therefore for any given 9 points in $\mathbb{P}^{2}$, there exists a (possibly reducible) cubic curve passing through these points. 
Through the paper, by abuse of notation, for an algebraic curve $C\subset \mathbb P^2$, the homogenous  defining polynomial will also be denoted by $C$.  
\begin{lem} \label{1}
Suppose that there exist irreducible cubics $C_{1}$ and $C_{2}$ such that
\begin{itemize} 
\item  $C_{1}$ contains $x_{1}, x_{2}, x_{3}, x_{4}, x_{5}, x_{6}, x_{7}, x_{8}, x_{9}$ 
\item $C_{2}$ contains $x_{1}, x_{2}, x_{3}, x_{4}, x_{5}, x_{6},  x_{10}, x_{11}, x_{12}$
\end{itemize} 
such that  $C_{1}$ does not contain $x_{10}, x_{11}, x_{12}$ and $C_{2}$ does not contain $x_{7}, x_{8}, x_{9}$. Then, $\mathcal{X}$ contains two polynomials $P, Q$ such that the common zero set of $P$ and $Q$ is discrete and contains $S$.  
\end{lem} 

\begin{proof}
Let $P_{1}=C_{1}C_{2}$. Then, it is clear that $P_{1}$ belongs to $\mathcal{X}$. Since $dim(\mathcal{X})\geq 4$, we can choose $P_{2}, P_{3}, P_ {4}\in \mathcal{X}$ such that the set $\{P_{1},P_{2},P_{3},P_{4}\}$ is linearly independent. If for some $i\in \{2,3,4\}$, neither  $C_{1}$ nor $C_{2}$ divide $P_{i}$, then we can take $P=P_{1}$ and $Q=P_{i}$: In virtue of the irreducibility of $C_{1}$ and $C_{2}$, it is clear that in this case $P$ and $Q$ will have no common components, so their common zero locus will be discrete; it will contain $S$ as well by construction. Also, if both $C_{1}$ and $C_{2}$ divide $P_{i}$ for some $i\in \{2,3,4\}$ then we obtain the contradiction $P_{1}=P_{i}$, so from now on we may assume that this does not happen. Suppose now that $C_{1}$ divides $P_{i}$ and $C_{2}$ divides $P_{j}$ where $i\neq j$ and $i,j\in \{2,3,4\}$. Then $P_{i}+P_{j}\in \mathcal{X}$ and neither $C_{1}$ nor $C_{2}$ divides $P_{i}+P_{j}$. Therefore, we can take $P=P_{1}, Q=P_{i}+P_{j}$ in this case, and we are done. 

The remaining case is when one of $C_{1}$ or $C_{2}$ divides $P_{2}, P_{3}, P_{4}$, and the other doesn't divide any of them. Suppose without loss of generality that $C_{1}$ divides $P_{2}, P_{3}$ and $P_{4}$. We can write $P_{2}=C_{1}D_{2}$, $P_{3}=C_{1}D_{3}$ and $P_{4}=C_{1}D_{4}$ where $D_{2}, D_{3}, D_{4}$ are (possibly reducible) cubics. Since $C_{1}$ does not contain $x_{10}, x_{11}, x_{12}$, each $D_{i}$ must contain these three points. For $j\in \{1,2,3,4,5,6\}$, if $x_{j}$ is a smooth point of $C_{1}$, then in order for $P_{2}, P_{3}$ and $P_{4}$ to have order $2$ at $x_{j}$, the cubics $D_{2}, D_{3}$ and $D_{4}$ must contain $x_{j}$ (this cannot be inferred if $x_{j}$ is a singular point of $C_{1}$). Since $C_{1}$ is an irreducible cubic by assumption, it can have at most one singular point. Therefore, without loss of generality we may assume that $x_{2},x_{3},x_{4},x_{5},x_{6}$ are smooth points of $C_{1}$, hence they belong to $D_{2}, D_{3}$ and $D_{4}$. 

We now have $3$ cubics $D_{2},D_{3},D_{4}$, each of which  intersects the irreducible cubic $C_{2}$ at a set of points containing the $8$ distinct points $x_{2},x_{3},x_{4},x_{5},x_{6},x_{10},x_{11},x_{12}$. By Bezout's theorem, the total intersection number of $D_{2}$ and $C_{2}$ is $9$. If there is a $9$th intersection point $x^{*}$ of $D_{2}$ and $C_{2}$ distinct from the previous $8$ points, then by  Cayley-Bacharach theorem, $D_{3}$ must also contain $x^{*}$. But then, by Noether's AF+BG theorem, $D_{3}$ must lie on the pencil through $C_{2}$ and $D_{2}$, namely there exists $[\alpha:\beta]\in \mathbb{CP}^{1}$ such that $D_{3}=\alpha C_{2}+\beta D_{2}$. This would imply that $P_{3}=\alpha P_{1}+\beta P_{2}$, contradicting the linear independence of $\{P_{1},P_{2},P_{3},P_{4}\}$. The only other possibility is that $C_{2}$ and $D_{2}$ have intersection number $2$ at one of the $8$ intersection points, without loss of generality at $x_{2}$. Then we can assume that the intersection numbers $\mu_{x_{2}}(C_{2},D_{3})$ and $\mu_{x_{2}}(C_{2},D_{4})$ are also  equal to $2$; else a generic member of the pencil through $D_{2}$ and one of $D_{3}$ or $D_{4}$ will intersect $C_{2}$ in $9$ distinct points, and we can repeat the argument in the beginning of this paragraph by using this generic member instead of $D_{2}$. 

To resolve this final case, note that $\mu_{x_{2}}(C_{2},D_{2})=2$ is the dimension of the $\mathbb{C}$-vector space $\mathcal{O}_{\mathbb{P}^2, x_{2}}/(C_{2},D_{2})$ where   $\mathcal{O}_{\mathbb{P}^2, x_{2}}$ denotes the local ring of the projective plane at $x_{2}$. Denoting by $m_{\mathbb{P}^{2},x_{2}}$ its maximal ideal, we observe that $m_{\mathbb{P}^{2},x_{2}}/(C_{2},D_{2})$ is $1$-dimensional. But now, both $D_{3}$ and $D_{4}$ contain $x_{2}$, hence their representatives in this local ring both belong to $m_{\mathbb{P}^{2},x_{2}}$. Hence, their equivalence classes in  $m_{\mathbb{P}^{2},x_{2}}/(C_{2},D_{2})$ will be linearly dependent. This implies that there exists $[\alpha:\beta]\in \mathbb{CP}^{1}$ such that the equivalence class of $\alpha D_{3}+\beta D_{4}$  is zero in all $\mathcal{O}_{\mathbb{P}^2, p}/(C_{2},D_{2})$, where $p$ ranges over all intersection points of $C_{2}$ and $D_{2}$. By Noether's AF+BG theorem again, we deduce that $\alpha D_{3}+\beta D_{4}$ lies on the pencil through $C_{2}$ and $D_{2}$. This contradicts the linear independence of $\{P_{1},P_{2},P_{3},P_{4}\}$ in this final case as well, finishing the proof. 
\end{proof} 

By $m_j(S)$, we denote the maximum number of points of $S$ which are contained in an algebraic curve of degree $j$. For a finite generic subset  $S \subset \mathbb C^2$  which contains at least $9$ points, $m_1(S)=2$, $m_2(S)=5$ and $m_3(S)=9$.

\begin{lem}\label{2}Let  $S\subset \mathbb C^2$ such that $|S|=12$ and $m_3(S)=9$. Then there exists a $u\in PSH(\mathbb C^2)$ of logarithmic growth which is locally bounded outside a finite set and satisfies one of the following:
\begin{itemize}

\item  $\gamma_u=6$ and $u$ has logarithmic poles at the points of $S$ of total weight $18$.  
\item $\gamma_u=5$ and $u$ has logarithmic poles at the points of $S$ of total weight $15$.
\item $\gamma_u=4$ and $u$ has logarithmic poles  at the points of $S$ of total weight $12$ .
\item $\gamma_u=3$ and $u$ has logarithmic poles  at the points of $S$ of total weight $9$ .
  
\end{itemize}

\end{lem}

\begin{proof}We will consider the different cases for the configurations of points of $S$. Since $m_3(S)=9$, we have $5 \leq m_2(S)\leq 7$. Also note that $m_{1}(S)\leq m_{2}(S)-2$. \\
\textbf{Case 1.}  $ m_2(S)=5$.  Consider the cubics $C_1$ and $C_2$ passing through  $$x_{1}, x_{2}, x_{3}, x_{4}, x_{5}, x_{6}, x_{7}, x_{8}, x_{9}$$ and $$x_{1}, x_{2}, x_{3}, x_{4}, x_{5}, x_{6}, x_{10}, x_{11}, x_{12}$$ respectively. Since $m_2(S)=5$, we have $m_1(S)\leq 3$. This implies that $C_1$ and $C_2$ are irreducible cubics. Since $m_3(S)=9$, $C_1$ and $C_2$ does not contain any other point of $S$. Then we choose $P,Q$ as in Lemma \ref{1}.  By Theorem \ref{cn}, $u=\frac{1}{2}\log(|P|^2+|Q|^2)$ is a plurisubharmonic function which has logarithmic poles  at the points of $S$ of total weight $18$, locally bounded outside a finite set and  $\gamma_u=6$. 

\textbf{Case 2.}  $ m_2(S)=6, m_1(S)\leq 3$. Consider the cubics $C_1$ and $C_2$ passing through  $$x_{1}, x_{2}, x_{3}, x_{4}, x_{5}, x_{6}, x_{7}, x_{8}, x_{9}$$ and $$x_{1}, x_{2}, x_{3}, x_{4}, x_{5}, x_{6}, x_{10}, x_{11}, x_{12}$$ respectively. If $m_1(S)=2$ then $C_1$ and $C_2$ are irreducible and we return Case 1.  Otherwise, we may assume that $m_1(S)=3$ and $C_1$ is reducible. Then $C_1=l_1\cup C$ where $l_1$ is a line through $3$ points of $\{x_1,\dots ,x_9\}$ and $C$ is a conic (possibly reducible) through the remaining $6$ points of $\{x_1,\dots ,x_9\}$.  Let $\mathcal{X}$ be the subspace of $\mathbb{C}[X,Y,Z]$ containing all homogeneous polynomials $P(X,Y,Z)$ of degree $6$ such that $P$ vanishes on all points of $S$, and  vanishes to second order at the points $x_{1},\ldots,x_{6}$. Then we take $P_1=C_1C_2\in \mathcal{X}$ and by counting the dimension of $\mathcal{X}$, we find another $P_2\in\mathcal{X}$ such that $P_1$ and $P_2$ are linearly independent.  If  $P_1$ and $P_2$ have no common factor then  $u=\frac{1}{2}\log(|P_1|^2+|P_2|^2)$ is a plurisubharmonic function which has logarithmic poles  at the points of $S$ of total weight $18$, locally bounded outside a finite set and  $\gamma_u=6$. 

Otherwise, we first assume that $P_2$ is divisible by $l_1$. Then $P_2=l_1\gamma$ where $\gamma$ is a curve of degree $5$ which passes through the $15$ points (counting multiplicity) of  $S\setminus l_1$ . If $\gamma$ and $CC_2$ have no common factor then  $u=\frac{1}{2}\log(|\gamma|^2+|CC_2|^2)$ is  a plurisubharmonic function which has logarithmic poles  at the points of $S$ of total weight $15$, locally bounded outside a finite set and  $\gamma_u=5$.  Otherwise, $\gamma$ is divisible by $C$ or $C_2$  or an irreducible factor of $C$ or $C_2$.  If $C$ divides $\gamma$ then $\gamma=C\gamma'$ where $\gamma'$  is a cubic which shares the same $9$ points of $S$ with $C_2$.   If $\gamma'$ and $C_2$ have no common factor then  $u=\frac{1}{2}\log(|\gamma'|^2+|C_2|^2)$ is  a plurisubharmonic function which has logarithmic poles  at the points of $S$ of total weight $9$, locally bounded outside a finite set and  $\gamma_u=3$.  Otherwise, $\gamma^{\prime}$ and $C_{2}$ either have a degree $1$ or degree $2$ common factor. If they have a degree $2$ common factor, then it must contain $6$ points of $S$, so the remaining lines share the same $3$ points, so they coincide, showing that $\gamma^{\prime}=C_{2}$. If they have a degree $1$ common factor, then it must contain $3$ points of $S$, so the remaining conics share the same $6$ points; these conics then coincide by applying Bezout's theorem (regardless of whether the conic is irreducible or reducible). In all cases, $P_{1}$ and $P_{2}$ become linearly dependent, causing a contradiction. A similar argument works if  a component of $C$ or $C_2$ divides $\gamma$.


In the case that $P_2$ is divisible by $C$ or a  component of $C$ or $C_2$ or a component of $C_2$,  the same discussion follows as above.



\textbf{Case 3.} $ m_2(S)=6$ and $m_1(S)=4$. Let $l$ be the unique line containing $4$ points for $S$, say $x_9,x_{10},x_{11},x_{12}.$ We consider the cubics $C_1$ and $C_2$ passing through  $$x_{1}, x_{2}, x_{3}, x_{4}, x_{5}, x_{6}, x_{7}, x_{9}, x_{10},$$ and $$ x_{1}, x_{2}, x_{3}, x_{4}, x_{5},x_{6}, x_{8}, x_{11}, x_{12}$$ respectively. If both $C_1$ and $C_2$ are irreducible then we return Case 1. Otherwise, we may assume that $C_1$ is reducible, say $C_1=l_1\cup C$ where $l_1$ is a line and $C$ is a conic. Since $m_2(S)=6$, the line $l_{1}$ must contain at least $3$ points out of the $9$ points of $S$ that $C_{1}$ contains, consequently $l_1\neq l$ since otherwise $l_{1}$ would also have to contain $x_{11}$ and $x_{12}$, contradicting $m_{1}(S)=4$. Since $l$ is the unique line containing $4$ points, $l_1$ contains exactly $3$ points of $x_{1}, x_{2}, x_{3}, x_{4}, x_{5}, x_{6}, x_{7}, x_{9}, x_{10}$  and $C$ contains the  remaining $6$ points. Then the result follows from the same argument as in Case 2.   

\textbf{Case 4.} $m_2(S)=7.$ By rearranging the points of $S$, we may assume that there is a conic $C$ through $x_1,\dots ,x_7.$ 

First we assume that $C$ is irreducible. Let $\tilde C$ be the conic passing through the points $$x_8,x_9,x_{10},x_{11}, x_{12}.$$  Since $m_3(S)=9$, $\tilde C$ is irreducible, otherwise the line component of $\tilde{C}$ containing at least $3$ points together with $C$ would give us a cubic containing at least $10$ points of $S$. Let $\mathcal{X}$ be the vector space of homogeneous polynomials of degree $4$, passing through the points of $S$. Then $P_1=C\tilde C\in \mathcal{X}$. The dimension of $\mathcal{X}$ is at least $15-12=3$, therefore we can choose $P_2\in \mathcal{X}$ such that $P_1$ and $P_2$ are linearly independent. It follows from Bezout's theorem that $P_1$ and $P_2$ have no common factor. Thus $u=\frac{1}{2}\log(|P_1|^2+|P_2|^2)$ is  a plurisubharmonic function which has logarithmic poles  at the points of $S$ of total weight $12$, locally bounded outside a finite set and  $\gamma_u=4$.

Now we consider the case where $C$ is reducible. We note that there cannot be any line containing $5$ or more points of $S$, since otherwise this line together with a conic through $5$ other poıints of $S$ would contradict $m_{3}(S)=9$. Therefore $C$ must split into two lines at least one of which has $4$ points. By rearranging the points, we may assume that $l$ is the line containing $4$ points, $x_1,x_2,x_3,x_4$ of $C$. Let us list all possible lines which contain $4$ points of $S$. Any such two lines must intersect at a point of $S$ since $m_2(S)=7$.   At most two such lines intersect at any point of $S$ since $m_3(S)=9$. Considering these restrictions, by rearranging the points of $S$, we may assume that there are at most $5$ lines, each containing $4$ points of $S$, (see Figure~\ref{fig1}). The proof below will be presented as if all $5$ lines are actually present, but the argument works without any change if there are less than $5$ such lines. 

\begin{figure}
\begin{center} 
\scalebox{0.56}
{\includegraphics{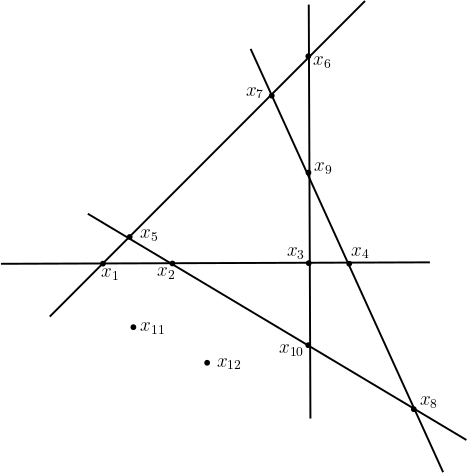}}
\caption{}\label{fig1} 
\end{center}
\end{figure}

Let $C_1$ and $C_2$ be the cubics through the points $$x_4,x_5,x_6,x_8,x_{11},x_{12}, x_{3},x_{7},x_{10}$$ and $$x_4,x_5,x_6,x_8,x_{11},x_{12}, x_{1},x_{2},x_{9},$$ respectively. Then $C_1$ and $C_2$ do not contain any $4$ collinear points of $S$. If $C_1$ and $C_2$ are irreducible then the result follows from Lemma \ref{1}. Otherwise, we may assume that $C_1$ is reducible. Let $C_1=C\cup l$ where $C$ is a conic which contains $6$ or $7$ points of $S$ and $l$ is a line. If $C$ contains $7$ points of $S$, then $C$ must be an irreducible conic and the result follows from the above argument. Otherwise, $C$ contains $6$ points of $S$ and we conclude the result by a similar argument to the one in Case 2. 
\end{proof}

 Before we prove our main result, we will eliminate some degenerate configurations of points in $V_{1/3}(T)$. 

\begin{pro}\label{prorem} Say $T$ is a positive closed $(1,1)$ current with unit mass  such that the set $V_{1/3}(T)$ is finite and  $V_{1/3}(T)\setminus C$ contains at least $2$ points for every cubic $C$.  Then, \\
(a)  $m_{1}(V_{1/3}(T))\leq 4$, \\
(b) $m_{2}(V_{1/3}(T))\leq 7$, \\ 
(c) An irreducible cubic can contain at most $9$ points of $V_{1/3}(T)$. 
\end{pro} 

\begin{proof} 
(a) Assume that  $V_{1/3}(T)$ contains at least  $5$ collinear  points. Say $x_1,\dots, x_{5}$ are contained in $L$. Then by Siu's decomposition theorem, \cite{Siu}, $T=a[L]+R$ for some current $R$ of bidegree $(1,1)$ having $0$ Lelong number generically on $L$. Using \cite{Dem}, for any $\epsilon>0$,  there exists a current $R'$ with $||R'||=||R||=1-a$, $\nu(R',x)=\nu(R,x)-\epsilon$ at any $x\in \mathbb P^2$ and $R'$ is smooth generically on $L$. Thus the intersection $R'\wedge [L]$ is well-defined. It follows that $$1-a=\int R'\wedge [L]\geq \displaystyle \sum_{i=1}^{5}\nu(R'\wedge [L],x_i)>\sum_{i=1}^{5}\nu(R',x_i)\nu([L],x_i)>5\left(\frac{1}{3}-a-\epsilon\right). $$
The second inequality follows from \cite[Corollary 5.10]{Dem} and the last inequality follows from the fact that $\nu([L],x_i)\geq 1$. Since $\epsilon>0$ is arbitrary, $a\geq\frac{1}{6}$. For any $x\in V_{1/3}(T)\setminus L$, 
$$\nu\left(\frac{R}{1-a},x\right)>\frac{2}{5}.$$ By \cite{Com}, all points (except possibly one point) in $V_{1/3}(T)\setminus L$  are contained in a conic $C$. Then $|V_{1/3}(T)\setminus (C\cup L)|\leq 1$, contradicting the assumption. Hence, $m_1(V_{1/3}(T))\leq 4$. 

(b) Assume that $V_{1/3}(T)$ contains at least  $8$  points on an irreducible conic $C$. Say $x_1,\dots, x_{8}$ are contained in $C$. Then by Siu's decomposition theorem, \cite{Siu}, $T=a[C]+R$ for some current $R$ of bidegree $(1,1)$  having $0$ Lelong number generically on $C$. Using \cite{Dem}, for any $\epsilon>0$,  there exists a current $R'$ with $||R'||=||R||=1-2a$, $\nu(R',x)=\nu(R,x)-\epsilon$ at any $x\in \mathbb P^2$ and $R'$ is smooth generically on $C$. It follows that $$2(1-2a)=\int R'\wedge [C]\geq \displaystyle \sum_{i=1}^{8}\nu(R'\wedge [C],x_i)>\sum_{i=1}^{8}\nu(R',x_i)\nu([C],x_i)>8\left(\frac{1}{3}-a-\epsilon\right). $$
Since $\epsilon>0$ is arbitrary, $a\geq\frac{1}{6}$. For any $x\in V_{1/3}(T)\setminus C$, 
$$\nu\left(\frac{R}{1-2a},x\right)>\frac{1}{2}.$$ By \cite{Com}, all points (except possibly one point) in $V_{1/3}(T)\setminus C$  are contained in a line $L$. Hence $|V_{1/3}(T)\setminus (C\cup L)|\leq 1$ , contradicting the assumption. 

Similar argument as above works if $C$ is  a union of $2$ lines, each of which contains $4$ points of $S$.  Let $C=L_1\cup L_2$ where $L_1$ and $L_2$  are the lines containing $x_1,\dots, x_4 $ and $x_5,\dots x_8$, respectively.  Then  $T=a[L_1]+b[L_2]+R$ for some current $R$ of bidegree $(1,1)$  having $0$ Lelong number generically on $C$. Using \cite{Dem}, for any $\epsilon>0$,  there exists a current $R'$ with $||R'||=||R||=1-a-b$, $\nu(R',x)=\nu(R,x)-\epsilon$ at any $x\in \mathbb P^2$ and $R'$ is smooth generically on $C$. It follows that

 \begin{eqnarray*}2(1-a-b)&=&\int R'\wedge ([L_1]+[L_2])\geq \displaystyle \sum_{i=1}^{8}\nu(R'\wedge ([L_1]+[L_2]),x_i)\\ &>&\sum_{i=1}^{4}\nu(R',x_i)\nu([L_1],x_i)+\sum_{i=5}^{8}\nu(R',x_i)\nu([L_2],x_i)>4\left(\frac{1}{3}-a-\epsilon+\frac{1}{3}-b-\epsilon\right). \end{eqnarray*}

Since $\epsilon>0$ is arbitrary, $a+b\geq\frac{1}{3}$. For any $x\in V_{1/3}(T)\setminus C$, 
$$\nu\left(\frac{R}{1-a-b},x\right)>\frac{1}{2}.$$ By \cite{Com}, all points (except possibly one point) in $V_{1/3}(T)\setminus C$  are contained in a line $L$. Hence $|V_{1/3}(T)\setminus (C\cup L)|\leq 1$ , contradicting the assumption.  Hence $m_2(V_{1/3}(T))\leq 7$. 

(c)  Assume that $V_{1/3}(T)$ contains at least $10$ points, $x_1,\dots, x_{10}$ on an irreducible cubic $C$. Then by Siu's decomposition theorem, \cite{Siu},  $T=a[C]+R$ for some current $R$ of bidegree $(1,1)$ having $0$ Lelong number generically on $C$ and $a\leq \frac{1}{3}$. Using  \cite{Dem}, for any $\epsilon>0$,  there exists a current $R'$ with $||R'||=||R||=1-3a$, $\nu(R',x)=\nu(R,x)-\epsilon$ at any $x\in \mathbb P^2$ and $R'$ is smooth generically on $C$.  It follows that $$3(1-3a)=\int R'\wedge [C]\geq \displaystyle \sum_{i=1}^{10}\nu(R'\wedge [C],x_i)>\sum_{i=1}^{10}\nu(R',x_i)\nu([C],x_i)>10\left(\frac{1}{3}-a-\epsilon\right). $$ 
 Since $\epsilon>0$ is arbitrary, $a\geq \frac{1}{3}$ and hence $T=\frac{1}{3}[C].$ This implies that $V_{1/3}(T)$ may contain only the singular point of $C$ and hence we conclude that at most $9$ points of $V_{1/3}(T)$ may lie on an irreducible cubic. 
\end{proof}

\begin{proof}(of Theorem \ref{0} ) We first assume that $E_{\beta}(T)$ contains an algebraic curve $C$ for some $\beta>\alpha$. By Siu's decomposition theorem, there exists a positive closed $(1,1)$ current $R$ on $\mathbb P^2$ with $0$ Lelong number generically on $C$ such that $T=\beta[C]+R$. Then $ \deg(C)\leq \frac{1}{\beta}<3$. If $C$   is a conic then $||R||=1-2\beta<\frac{1}{3}$. This implies that $V_{\alpha}(T)=C$. If $C$ is a line, then $||R||=1-\beta<\frac{2}{3}.$ It follows that $V_{\alpha}(T)=C\cup V_{\alpha}(R)$. By Theorem 1.1 in \cite{Com}, there exist a line $L$ and a point $p$ such that  $$V_{\alpha}(R)\subset V_{\frac{1}{2}}\left(\frac{R}{||R||}\right)\subset L\cup \{p\} .$$ Then $|V_{\frac{1}{3}}(T)\setminus (L\cup C)|\leq 1$.

Now we assume that $E_{\beta}(T)$ is $0$ dimensional for all $\beta>\alpha$. Then $V_{\alpha}(T)$ is at most countable. If  $V_{\alpha}(T)$ is infinite then $E_{\frac{1}{3}}(T) $ contains an algebraic curve $C$ of degree at most $3$ and $T=\frac{1}{3}[C]+R$. If $C$ is a cubic then $R=0$. This implies that $\nu(T,p)>\frac{1}{3}$ only at the singular points of $C$ and hence $V_{\frac{1}{3}}(T)$ is finite. Hence $C$ must be a conic or line. If $C$ is a conic then  $||R||=\frac{1}{3}$ and $V_{\alpha}(T)\subset C$. If $C$ is a line then $||R||=\frac{2}{3}$. Since $V_{\alpha}(R)\subset V_{\frac{1}{2}}\left( \frac{R}{||R||} \right)$,  Theorem 1.1 in \cite{Com}  implies that there exists a line $L$ such that $|V_{\alpha}(R)\setminus L|\leq 1$. Thus $|V_{\alpha}(T)\setminus (L\cup C)|\leq 1$. 

Finally we assume that $V_{\alpha}(T)$ is finite and $|V_{\alpha}(T)\setminus C|>2$ for any cubic $C$. Then $V_{\alpha}(T)$ has a subset $S$ with $|S|=12$. 

By Proposition \ref{prorem}, $m_1(S)\leq 4$ and $m_2(S)\leq 7$. We also note that $m_3(S)\leq 11:$ If a cubic contains all the points of $S$, then it must be reducible by  Proposition \ref{prorem}. But this contradicts to the fact that $m_1(S)\leq 4$ and $m_2(S)\leq 7$. 

\textbf{Case 1.} $m_3(S)=9$. By applying Proposition \ref{prop1}  with  one of the plurisubharmonic functions $u$ in Lemma \ref{2}, we obtain that $\alpha <\frac{1}{3}$ which contradicts to the choice of $\alpha$. 

Now we move to the cases where $m_3(S)=10$. Let $\Gamma$ be a cubic containing $10$ points of $S$. By above discussion, $\Gamma$ must be reducible,  $m_2(S)\leq 7$ and $m_1(S)\leq 4$. Thus $m_2(S)=6$ or $7$. First we assume that $m_2(S)=6.$ \\

\textbf{Case 2.} $m_3(S)=10$, $m_2(S)=6.$ Then $m_{1}(S)=4$ and there is a unique line containing $4$ points of $S$. Let $L$ be the line containing the points $x_{1},x_{2},x_{3},x_{4}$, $C_1$ and $C_{2}$ be the cubics through $$x_{1},x_{2},x_{5},x_{7},x_{8},x_{9},x_{10},x_{11},x_{12}$$ and $$x_{3},x_{4},x_{6},x_{7},x_{8},x_{9},x_{10},x_{11},x_{12}$$   
respectively. 
If $C_1$ and $C_2$ are irreducible, then by Proposition \ref{prorem}, they can not contain any other point of $S$. In this case, we choose the polynomials $P,Q\in \mathcal X$ as in  Lemma \ref{1} and define $u=\frac{1}{2}\log(|P|^2+|Q|^2)$. Then $u$ satifies the first weight condition in Lemma \ref{2} and  result follows as in Case 1 above.

Hence we may assume that $C_1=l_1\cup C'$ where $l_1$ is a line and $C'$ is a conic. Since $m_3(S)=10$, $l_1\neq L.$ Then $l_1$ contains $3$ points of $S$ and $C'$ contains the remaining $6$ points of $C_1$. We have a similar decomposition for $C_2$. Then by a similar argument as in Case 2 of Lemma \ref{2}, we obtain  a plurisubharmonic function $u$ which satisfies one of the weight conditions in Lemma \ref{2}. It follows from  Proposition \ref{prop1} with $u$ that $\alpha<\frac{1}{3}$ which contradicts to the choice of $\alpha$. \\

\textbf{Case 3.} $m_3(S)=10$, $m_2(S)=7.$ Then $m_1(S)=3$ or $4$. Let $C$ be a conic with $7$ points, $x_1,\dots, x_{7}$ and $\tilde C$ be the conic through $x_8,\dots, x_{12}$.  First we assume that $C$ is irreducible. If $\tilde C$ is also irreducible, then we take $P_1=C\tilde C$ and $P_2$ a polynomial of degree $4$ which passes through the points of $S$ and linearly independent with $P_1$. It follows from Bezout's theorem that $P_1$ and $P_2$ have no common factor. Then $u=\frac{1}{2}\log (|P_1|^2+|P_2|^2)$ is a plurisubharmonic function with total weight $12$ at the points of $S$, $\gamma_u=4$ and locally bounded outside $S$.  Proposition \ref{prop1} with $u$ implies  that $\alpha<\frac{1}{3}$ which contradicts to the choice of $\alpha$. 

Thus, by rearranging the points of $S$,  we may assume that $\tilde C$ is reducible and  $x_{8},x_{9},x_{12}$ are collinear. We consider the conics $C_i$ passing through the points $x_8,x_9,x_{10},x_{11},x_{i}$ for $i=1,\dots,7.$ If $C_i$ is irreducible for some $i$, then we take $P_1=CC_i$ and a polynomial $P_2$ of degree $4$ which vanishes at the points of $S\setminus \{x_{12}\}$ and vanishes to second order  at $x_i$. Then it follows from Bezout's theorem that $P_1$ and $P_2$ have no common factor. Hence    
$u=\frac{1}{2}\log (|P_1|^2+|P_2|^2)$ is a plurisubharmonic function with total weight $12$ at the points of $S\setminus \{x_{12}\}$, $\gamma_u=4$ and locally bounded outside $S$. This implies with Proposition \ref{prop1} that $\alpha<\frac{1}{3}$. 

Now we assume that $C$ is irreducible and $C_i$ is reducible for all $i:1,\dots, 7.$ By $L_{i,j}$ we denote the line through $x_i$ and $x_j$.   Since $m_3(S)=10$ and   $x_8,x_9,x_{12}$ are collinear, the line $L_{8,9}$ does not contain either of the points $x_{10}$ or $x_{11}$. For the moment, we also assume that $L_{10,11}$ does not contain $x_8,x_9,x_{12}$. Then $m_1\{x_8,x_9,x_{10},x_{11}\}=2$, each $x_i$, $i=1,\dots,7$, belongs to $L_{jk}$ for some $8 \leq j,k\leq 11.$ 

By (P1), we refer to the property that $L_{8,9}$ intersects $C$ only at $x_7$ and $L_{10,11}$  possibly intersects $C$ only at $x_7$. Then we take $P_1=CL_{8,9}L_{10,11}$ and $P_2$ is a polynomial of degree $4$ linearly independent with $P_{1}$, which vanishes at the points of $S\setminus\{x_{12}\}$ and vanishes to second order at $x_7$. Then $P_1$ and $P_2$ have no common factor. Hence $u=\frac{1}{2}\log (|P_1|^2+|P_2|^2)$ is a plurisubharmonic function with total weight $12$ at the points of $S\setminus\{x_{12}\}$, $\gamma_u=4$, $u$ is locally bounded outside $S$ and we are done.

By (P2), we refer to the property that $L_{8,9}$ intersects $C$ only at $x_7$ and $L_{10,11}$   intersects $C$ only at $x_6$. Then we take $P_1=CL_{8,9}L_{10,11}$ and $P_2$ is a polynomial of degree $4$  linearly independent with $P_{1}$, which vanishes at the points of $S\setminus\{x_{12}\}$ and vanishes to second  order at $x_7$. Then $P_1$ and $P_2$ have no common factor and the result follows as above.

By (P3), we refer to the property that $L_{8,10}$ intersects $C$  at $x_6,x_7$ and $L_{9,11}$   intersects $C$ only at $x_5$. Then we take $P_1=CL_{8,10}L_{9,11}$ and $P_2$ is a polynomial of degree $4$  linearly independent with $P_{1}$, which vanishes  at the points of $S\setminus\{x_{12}\}$ and vanishes to second order at $x_5$. Then $P_1$ and $P_2$ have no common factor and the result follows as above.

Note that in $(P1),(P2),(P3)$,  we may replace the point $x_9$ with $x_{12}$ and the same conclusion holds. Thus we may assume that the configurations $(P1)$, $(P2)$ and $(P3)$ do not hold for the lines $L_{jk}$ where $j,k\in\{8,9,10,11\}$ and $j,k\in \{8,10,11,12\}$.

Consider the pairs of lines $(L_{8,9},L_{10,11}), (L_{8,10},L_{9,11}), (L_{8,11},L_{9,10}).$ 
Then each pair contains at most $3$ points of $\{x_1,\dots ,x_7\}\subset C$ since $m_2(S)=7$. Since $m_1\{x_8,x_9,x_{10},x_{11}\}=2$, each point $x_i$, $i=1,\dots,7$, belongs to a unique pair of lines. If the pair $(L_{8,9},L_{10,11})$ contains only one point, say $x_7$ of $C\cap S$ then property (P1) holds. Otherwise, since $m_2(S)=7$,  we may assume that the pair $(L_{8,9},L_{10,11})$ contains $2$ points, say $x_1,x_2$. If $x_1$ and $x_2$ lie on different lines, then property (P2) holds. Thus we may assume that $x_1,x_2\in L_{10,11}$. Note that one of these points may belong to $L_{8,9}$ as well. One of the remaining pairs of lines, say $(L_{8,10},L_{9,11})$ must contain $3$ points of $C\cap S$. Assuming that (P3) does not hold, we may assume that $x_6,x_7\in L_{8,10}$ and $x_5,x_7\in L_{9,11}$. We consider the similar arrangement for the lines  $L_{jk}$ where $j,k\in \{8,10,11,12\}$. Then either the pair $(L_{8,11},L_{10,12})$ or $(L_{8,10}, L_{11,12})$  must contain $3$ points of $C\cap S$.

First we assume that the first pair $(L_{8,11},L_{10,12})$ contains  $3$ points of $C\cap S$. As (P3) does not hold, we may assume that $x_3,x_4\in L_{8,11}$ and $x_4,x_5\in L_{10,12}$.  Then each line $L_{9,10}$ and $L_{11,12}$ may contain only one point of $C\cap S$, say $x_3$ and $x_6$ respectively. But then (P2) holds for the lines $L_{9,10}$ and $L_{11,12}$.

Similarly, we assume now  that  the other pair $(L_{8,10},L_{11,12})$ contains  $3$ points of $C\cap S $.
Say $x_6,x_7\in L_{8,10}$ and $x_4,x_6\in L_{11,12}$. Since (P1) and (P2) do not hold for the pair $(L_{8,11},L_{10,12})$,  $x_3\in L_{10,12}$ and $x_5\in L_{10,12}\cap L_{9,11}.$ As $x_3,x_9,x_{10},x_{12}$ are not collinear, $x_3\notin L_{9,10}$. Then  $x_4\in L_{9,10}$ and (P1) or (P2) (if $x_3\in L_{8,11}$)  holds for the pair $(L_{9,10},L_{8,11})$. Hence we are done in the case that $C$ is irreducible and $L_{10,11}$ does not contain any of the points $x_8,x_9,x_{12}$. 

Now we assume that the line $L_{10,11}$ contains $x_{12}$. Since $m_1(S)\leq 4$, each line $L_{8,9}$ and $L_{10,11}$ contains at most $1$ point of $C\cap S$ and hence (P1) or (P2) holds for the pair $(L_{8,9},L_{10,11})$.

Finally, we may assume that there does not exist any irreducible conic containing $7$ points of $S$. Then $m_1(S)=4$. We will list all possible lines containing $4$ points of $S$. Since $m_2(S)=7$, any such two lines must intersect at a point of $S$. Since $|S|=12$,  there are at most $3$ such lines passing through at any point $x_i \in S$. 

First we consider the case that there does not exist any point in $S$ through which $3$ such lines pass. Then considering the above restrictions, there are at most $5$ lines containing $4$ points of $S$ (see Figure~\ref{fig2}).
\begin{figure}
\begin{center} 
\scalebox{0.5}
{\includegraphics{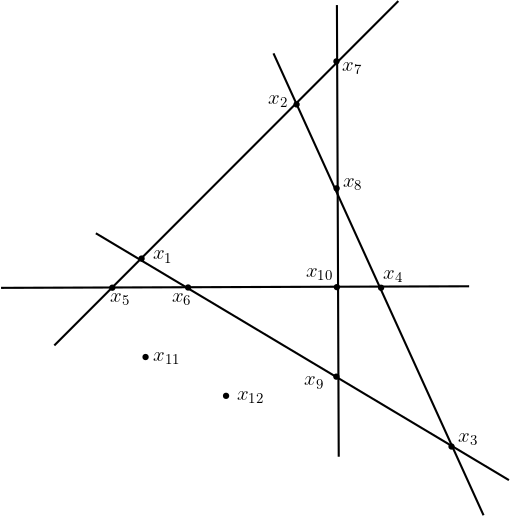}}
\caption{}\label{fig2} 
\end{center}
\end{figure}

We consider the cubics $C_1$ and $C_2$ passing through the points of $S\setminus \{x_2,x_6,x_{10}\}$ and $S\setminus\{x_5,x_8,x_9\}$. Then the result follows by a similar argument to the one in Case 4 of Lemma \ref{2}. 

Now we consider the case that there is only one point, say $x_1$, through which $3$ lines containing $4$ points of $S$ pass.  We may assume that the points in the sets $\{x_1,x_2,x_3,x_4\}$,        $\{x_1,x_5,x_6,x_7\}$ and $\{x_1,x_8,x_9,x_{10}\}$ are collinear. Then there are at most $5$ lines  containing $4$ points of $S$ (see Figure ~\ref{fig3}).
\begin{figure}
\begin{center} 
\scalebox{0.55}
{\includegraphics{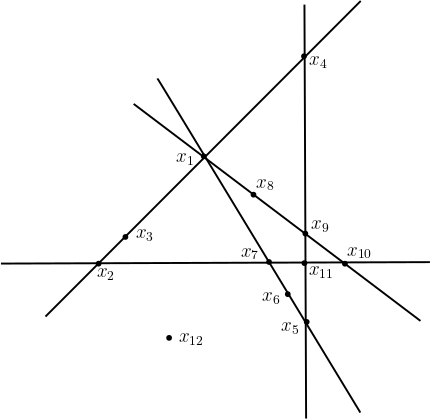}}
\caption{}\label{fig3} 
\end{center}
\end{figure}
We consider the cubics  $C_1$ and $C_2$ passing through the points of $S\setminus \{x_1,x_4,x_{11}\}$ and $S\setminus\{x_2,x_5,x_{10}\}$. Then the result follows by a similar argument to the one in Case 4 of Lemma \ref{2}.

Finally, we consider the case that there is more than one point through which $3$ lines containing $4$ points of $S$ pass.  Then, other than $x_1$, there is a point of $S$ through which $3$ lines containing $4$ points of $S$ pass. This point can be a point on a line through $x_1$ or a point outside of the $3$ lines through $x_1$. 

In the first case, we may take this point as $x_2$ (see Figure ~\ref{fig4}).
\begin{figure}
\begin{center} 
\scalebox{0.55}
{\includegraphics{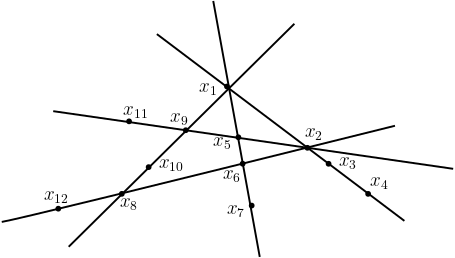}}
\caption{}\label{fig4} 
\end{center}
\end{figure}
It is easy to check that there are no $4$ collinear points in the sets $S\setminus \{x_2,x_7,x_{10}\}$ and $S\setminus\{x_1,x_{11},x_{12}\}$. Consider the cubics  $C_1$ and $C_2$ passing through the points of $S\setminus \{x_2,x_7,x_{10}\}$ and $S\setminus\{x_1,x_{11},x_{12}\}$. Then the result follows by a similar argument to the one in Case 4 of Lemma \ref{2}.

In the second case, we may assume that $3$ lines containing $4$ points of $S$ pass through $x_{11}$ (see Figure ~\ref{fig5}).
\begin{figure}
\begin{center} 
\scalebox{0.52}
{\includegraphics{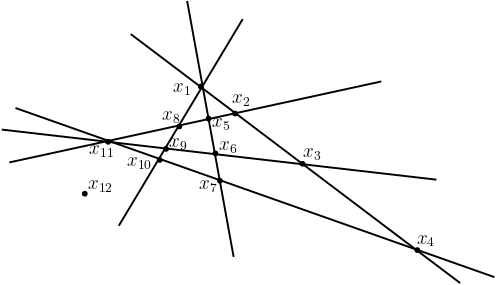}}
\caption{}\label{fig5} 
\end{center}
\end{figure}
If there does not exist any other line containing $4$ points of $S$, then we take the cubics $C_1$ and $C_2$ passing through the points of $S\setminus \{x_4,x_5,x_{9}\}$ and $S\setminus\{x_1,x_{11},x_{12}\}$. Then the result follows by a similar argument to the one in Case 4 of Lemma \ref{2}. Otherwise, we may assume that $x_2,x_6,x_{10},x_{12}$ are collinear and return to the first case.

\textbf{Case 4.} $m_3(S)=11$. In this case, $m_1(S)=4$ and $m_2(S)=7$. Let $\Gamma$ be the cubic with $11$ points of $S$. Then by Proposition \ref{prorem}, $\Gamma$ must be reducible, say $\Gamma=C\cup L$.  $C$ must be irreducible, otherwise one of the components of $C$ together with $L$ would contain $8$ points of $S$ which contradicts to $m_2(S)=7$.  
We may assume that the points $x_1,\dots,x_7$ are contained in $C$ and $x_8,\dots,x_{11}$ are contained in $L$. Since we assumed that $|V_{1/3}(T)\setminus \gamma|>1$ for any cubic  $\gamma$,  there is a point $p\in V_{1/3}(T)\setminus (\Gamma\cup x_{12})$. Let $L_{p,12}$ denote the line which passes through $p$ and $x_{12}$. First we assume that $L_{p,12}\cap L\cap S=\emptyset.$ We take $P_1=CL_{p,12}L$ and $P_2$  a polynomial of degree $4$, linearly independent with $P_{1}$, through the points of $S\cup \{p\}$. It follows from Bezout's theorem that $P_1$ and $P_2$ have no common factor. Thus $u=\frac{1}{2}\log (|P_1|^2+|P_2|^2)$ is a plurisubharmonic function with total weight $13$ at the points of $S\cup \{p\}$  locally bounded outside $S$ and $\gamma_u=4$. Proposition \ref{prop1} with $u$ implies that $\alpha<\frac{4}{13}$ which contradicts to the choice of $\alpha$.\\

Now we consider the case that $L_{p,12}$ intersects $L$ at a point of  $S$ and $L_{p,12}\cap C \cap S=\emptyset.$ Then the same polynomials $P_1$ and $P_2$ above work. \\

Finally, we consider the case that $L_{p,12}\cap L\cap S\neq \emptyset$ and $L_{p,12}\cap C\cap S\neq \emptyset$. Say $x_1\in L_{p,12}\cap C$ and $x_{11}\in L_{p,12}\cap L$. Then take $P_1=CLL_{p,12}$ and $P_2$ a polynomial of degree $4$, linearly independent with $P_{1}$,  which vanishes at the points of $(S\cup\{p\})\setminus \{x_{10},x_{11}\}$ and vanishes to the second order  at $x_1$. Then $P_1$ and $P_2$ have no common factor and the result follows as above.\\

\begin{rem} We should note that a slightly better conclusion follows from the proof of  Theorem \ref{0}. To get a contradiction in the proof, we assumed that $V_{1/3}(T)$ contains a subset $S$ with $12$ points and $|V_{1/3}(T) \setminus C|> 1 $ for any cubic $C$. Hence if   $V_{1/3}(T)$ is finite then $|V_{1/3}(T)|=11$ or $|V_{1/3}(T) \setminus C|\leq 1 $ for some cubic $C$.   
\end{rem}

\end{proof}

\noindent \textbf{Acknowledgments.} We would like to thank to the referee for his/her careful reading and suggestions. The second author is supported by T\"UB\.ITAK 3501 Proj. No 120F084.

\end{document}